\documentclass[a4paper,11pt]{amsart}
\usepackage{amssymb,amsmath,amsthm}
\usepackage{amsfonts}
\usepackage{layout}
\def\ds{\displaystyle}

\def\bb{\overline{b}}
\def\eee{\qquad\mbox{and}\qquad}

\newcommand{\leg}[2]{\left(\frac{#1}{#2}\right)}
\newcommand{\qbin}[2]{\genfrac{[}{]}{0pt}{}{#1}{#2}_q}
\newcommand{\qq}[1]{[#1]_q}

\newtheorem{thm}{Theorem}[section]
\newtheorem{lem}[thm]{Lemma}

\begin{document}

\title[Some $q$-analogs of congruences for central binomial sums]{\bf Some $q$-analogs of congruences \\ for central binomial sums}

\author{Roberto Tauraso}
\email{tauraso@mat.uniroma2.it}
\address{Dipartimento di Matematica, 
Universit\`a di Roma ``Tor Vergata'', 
via della Ricerca Scientifica, 
00133 Roma, Italy}



\date{}

\maketitle
\begin{abstract}
\noindent We establish $q$-analogs for four congruences involving central binomial coefficients.
The $q$-identities necessary for this purpose are shown via the $q$-WZ method.
\end{abstract}

\section{Introduction}

Recently, a number of papers have appeared concerning congruences for central binomial sums (see through the references). 
Here we would like to draw the attention to one aspect of the matter
which has been partly neglected so far: $q$-analogs.
In \cite{GZ,Ta1}, the authors identified a first group of such 
congruences which have a $q$-counterpart. 
Among them we mention: for any prime $p>2$
$$\begin{array}{ll}
\ds \sum_{k=0}^{p-1}\binom{2k}{k}\equiv\leg{p}{3} \pmod{p^2}, 
&\ds\sum_{k=0}^{p-1}q^{k}\qbin{2k}{k}\equiv
\leg{p}{3}q^{\lfloor {p^2\over 3}\rfloor}\pmod{\qq{p}},\\
\ds\sum_{k=0}^{p-1}(-1)^k\binom{2k}{k}\equiv \leg{p}{5} \pmod{p},
&\ds\sum_{k=0}^{p-1}(-1)^kq^{-{k+1\choose 2}}\qbin{2k}{k}\equiv
\leg{p}{5}q^{-\lfloor {p^4\over 5}\rfloor} \pmod{\qq{p}},\\
\ds\sum_{k=0}^{p-1}{1\over 2^k}\binom{2k}{k}\equiv (-1)^{{p-1\over 2}} \pmod{p^2},
&\ds\sum_{k=1}^{p-1} {q^{k}\over (-q;q)_{k}}\qbin{2k}{k}\equiv (-1)^{{p-1\over 2}}q^{\lfloor {p^2\over 4}\rfloor}\pmod{\qq{p}},
\end{array}$$
where $\leg{\cdot}{\cdot}$ denotes the Legendre symbol.
It has been conjectured 
in \cite{GZ} that the first $q$-congruence holds modulo $\qq{p}^2$, and 
we claim that the same can be said for the third one. 
However, in this short note, we are not going to refine these $q$-congruences.
Instead, we will present a few more examples of this phenomena.
More precisely we show that the congruences
\begin{align}
&\label{pc1}\sum_{k=1}^{p-1} {(1/2)^k \over k}\binom{2k}{k}\equiv-{3\over 2}\sum_{k=1}^{p-1} {(-1/2)^k \over k}\binom{2k}{k}^{-1}\equiv Q_p(2) \pmod{p},\\
&\label{pc2}\sum_{k=1}^{p-1} {1\over  k}\binom{2k}{k}\equiv 0 \pmod{p^2},\\
&\label{pc3}{5\over 2}\sum_{k=1}^{p-1} {(-1)^k\over  k^2}\binom{2k}{k}\equiv -\sum_{k=1}^{p-1} {1\over  k^2} \pmod{p^3},
\end{align}
where $p>5$ is a prime and $Q_p(2)=(2^{p-1}-1)/p$ is the usual Fermat's quotient, 
have as $q$-analogs respectively
\begin{align}
&\label{qc1}\sum_{k=1}^{p-1} {q^k\over (-q;q)_{k}\qq{k}}\qbin{2k}{k}\equiv
-{1\over 2}\sum_{k=1}^{p-1} {(-1)^k(1+q^k+q^{2k})\over  (-q;q)_{k}\qq{k}}\qbin{2k}{k}^{-1}\equiv Q_p(2;q)\pmod{\qq{p}},\\
&\label{qc2}\sum_{k=1}^{p-1} {(1+q^k+q^{2k})\,q^{-{k\choose 2}}\over (1+q^k)^2\qq{k}}\qbin{2k}{k}\equiv {\qq{p}(p^2-1)(1-q)^2\over 24}\pmod{\qq{p}^2},\\
&\label{qc3}\sum_{k=1}^{p-1} {(-1)^k(1+3q^k+q^{2k})\,q^{-{k\choose 2}}\over (1+q^k)\qq{k}^2}\qbin{2k}{k}\equiv -\sum_{k=1}^{p-1} {q^k\over \qq{k}^2}
-{\qq{p}^2(p^4-1)(1-q)^4\over 240}\pmod{\qq{p}^3}.
\end{align}
where $Q_p(2;q)=((-q;q)_{p-1}-1)/\qq{p}$.
Proofs of \eqref{pc1}, \eqref{pc2}, \eqref{pc3} can be found in \cite[Theorem 3.1]{Ta2} (\eqref{pc2} appeared first in \cite{ST2}).

We are optimistically hopeful that there are plenty of interesting $q$-analogs
to discover. For example, recently in \cite{PP}, the authors proved that for $0<q<1$
$$\sum_{k=1}^{\infty} {(1+2q^k)\,q^{k^2}\over \qq{k}^2}\qbin{2k}{k}^{-1}
=\sum_{k=1}^{\infty} {q^k\over \qq{k}^2}.$$
By letting $q\to 1$, it gives a well known series identity which then happened to have
a congruence version: in \cite{Ta2} we showed that for any prime $p>3$
$$\sum_{k=1}^{p-1} {1 \over k^2}\binom{2k}{k}^{-1}\equiv -{1\over 6}\sum_{k=1}^{p-1} {1\over  k^2}\pmod{p^3}.$$
Is there a $q$-analog for the above congruence?

\section{Notations a preliminary results}

The first two results of this section yield a family of $q$-analogs of the classical 
congruence for the harmonic sums: for any prime $p>d+2$ where $d$ is a positive integer,
$$H_{p-1}(d):=\sum_{k=1}^{p-1}{1\over k^d}\equiv 
\left\{
\begin{array}{ll}
0 \pmod{p^2} &\mbox{if $d$ is odd,}\\
0 \pmod{p}   &\mbox{if $d$ is even.}
\end{array}
\right.$$
This family depends on two integer parameters $a,b$ and it concerns the sum 
$$\sum_{k=1}^{p-1}{q^{bk}\over \qq{ak}^d}$$
where 
$$\qq{n}={1-q^n \over 1-q}=1+q+\dots+q^{n-1}.$$
Several special cases have already been discussed by numerous authors (see \cite{An2,Di,Pa,SP}).
In particular, K. Dilcher found in \cite{Di} a determinant expression in the case when $a=1$ and $b\in\{0,1\}$.
We point out that, in this paper, two rational functions in $q$ are congruent modulo $\qq{p}^r$ for $r\geq 1$
if the numerator of their difference is congruent to $0$ modulo $\qq{p}^r$ in the polynomial ring $\mathbb{Z}[q]$
and the denominator is relatively prime to $\qq{p}$.

\begin{thm} For any prime $p>2$, if $a,b,d$ are integers such that $a,d>0$, $b\geq 0$ and $\gcd(a,p)=1$ then 
\begin{equation}\label{qH}
\sum_{k=1}^{p-1}{q^{bk}\over \qq{ak}^d}\equiv
{(1-q)^d\over p^d}\left((-1)^d p
\sum_{s=0}^{d-1}c_s{r_0+sp\choose d}-\sum_{s=0}^d (-1)^s{d \choose s}{sp \choose 2d}\right)
\pmod{\qq{p}}
\end{equation}
where $r_0\equiv -b/a \pmod{p}$ such that $r_0\in \{0,1,\dots,p-1\}$ and
$$c_s=\sum_{k=0}^s (-1)^{s-k}{r_0+kp+d-1\choose d-1}{d\choose s-k}.$$
\end{thm}
\begin{proof} Let $q$ be a $p$-root of unity such that $q\neq 1$. Since
$$\sum_{k=1}^{p-1}q^{k(aj+b)}=-1+
\left\{
\begin{array}{cl}
p &\mbox{if $p\mid (aj+b)$},\\
0 &\mbox{otherwise},
\end{array}
\right.
$$
it follows that
\begin{align*}
\sum_{k=1}^{p-1}{q^{bk}\over (1-q^{ak}z)^d}
&=\sum_{k=1}^{p-1}q^{bk}\sum_{j\geq 0} {j+d-1\choose d-1}q^{akj}z^j\\
&=\sum_{j\geq 0} {j+d-1\choose d-1}z^j\sum_{k=1}^{p-1}q^{k(aj+b)}\\
&=p\sum_{s\geq 0} {r_0+sp+d-1\choose d-1}z^{r_0+sp}-{1\over (1-z)^d}\\
&={p\sum_{s=0}^{d-1} c_s z^{r_0+sp}\over (1-z^p)^d}-{1\over (1-z)^d}.
\end{align*}
Let $z=1+w$, then 
\begin{align*}
\sum_{k=1}^{p-1}{q^{bk}\over (1-q^{ak})^d}
&=\lim_{w\to 0}
{p\sum_{s=0}^{d-1} c_s (1+w)^{r_0+sp} -\left({1-(1+w)^p\over -w}\right)^d\over (1-(1+w)^p)^d}\\
&=\lim_{w\to 0}
{p\sum_{s=0}^{d-1} c_s {r_0+sp\choose d}w^d+o(w^d)
 -(-1)^d\sum_{s=0}^{d}(-1)^s{d \choose s}{sp \choose 2d}w^d+o(w^d)
\over (-pw+o(w))^d}\\
&={1\over p^d}\left((-1)^d p
\sum_{s=0}^{d-1}c_s{r_0+sp\choose d}-\sum_{s=0}^d (-1)^s{d \choose s}{sp \choose 2d}\right).
\end{align*}
\end{proof} 

\noindent  The following special cases are worth mentioning. By letting $d=1,2,3$ in \eqref{qH}, we obtain these $q$-congruences 
which hold modulo $\qq{p}$:
\begin{align}
\label{qH1} &\sum_{k=1}^{p-1}{q^{bk}\over \qq{ak}}\equiv {(1-q)}\left({p-1\over 2}-r_0\right), \\
\label{qH2} &\sum_{k=1}^{p-1}{q^{bk}\over \qq{ak}^2}\equiv {(1-q)^2}\left(-{(p-1)(p-5)\over 12}+{r_0(p-2-r_0)\over 2}\right), \\
\label{qH3} &\sum_{k=1}^{p-1}{q^{bk}\over \qq{ak}^3}\equiv {(1-q)^3}\left(-{(p-1)(p-3)\over 8}-{r_0(p^2-9p+12-3r_0(p-3)+2r_0^2)\over 12}\right). 
\end{align}

\begin{thm} Let $b,\overline{b},a,d$ be non-negative integers such that $ad=b+\bb>0$. 

\noindent Then for any prime $p>2$ such that $\gcd(a,p)=1$,  
\begin{equation}\label{qHpos}
\sum_{k=1}^{p-1}{(-1)^{d-1}q^{bk}+q^{\bb k}\over \qq{ak}^d}\equiv 
b(1-q)\qq{p}\sum_{k=1}^{p-1}{q^{\bb k}\over \qq{ak}^d}-ad\qq{p}\sum_{k=1}^{p-1}{q^{\bb k}\over \qq{ak}^{d+1}}
\pmod{\qq{p}^2}
\end{equation}
and 
\begin{equation}\label{qHneg}
\sum_{k=1}^{p-1}{(-1)^k((-1)^{d}q^{bk}+q^{\bb k})\over \qq{ak}^d}\equiv 
b(1-q)\qq{p}\sum_{k=1}^{p-1}{(-1)^kq^{\bb k}\over \qq{ak}^d}-ad\qq{p}\sum_{k=1}^{p-1}{(-1)^kq^{\bb k}\over \qq{ak}^{d+1}}\pmod{\qq{p}^2}.
\end{equation}
Moreover, let $b_1,\bb_1,b_2,\bb_2$ be non-negative integers.

\noindent If $d_1=b_1+\bb_1>0$ and $d_2=b_2+\bb_2>0$ then
\begin{equation}\label{qHH}
\sum_{1\leq j<k\leq p-1}{q^{b_1 j+b_2 k}+(-1)^{d_1+d_2}q^{\bb_1 j+\bb_2 k}\over \qq{j}^{d_1}\qq{k}^{d_2}}\equiv
\sum_{k=1}^{p-1}{q^{b_1 j}\over \qq{j}^{d_1}}\cdot \sum_{k=1}^{p-1}{q^{b_2 k}\over \qq{k}^{d_2}}-\sum_{k=1}^{p-1}{q^{(b_1+b_2) k}\over \qq{k}^{d_1+d_2}}
\pmod{\qq{p}}.
\end{equation}
\end{thm}
\begin{proof}As regards \eqref{qHpos} and \eqref{qHneg}, 
it suffices to note that
\begin{align}
{(-1)^{d}q^{b(p-k)}\over \qq{a(p-k)}^d}&={(-1)^dq^{b(p-k)+adk}\over (\qq{ap}-\qq{ak})^d}={q^{b p+\bb k}\over \qq{ak}^d(1-\qq{ap}/\qq{ak})^d}\nonumber \\
&\equiv {q^{\bb k}(1-b(1-q)\qq{p})\over \qq{ak}^d}\left(1+d{\qq{p}\over \qq{ak}}\right)\nonumber \\
&\label{qcp}\equiv {q^{\bb k}\over \qq{ak}^d}-{b(1-q)\qq{p} q^{\bb k}\over \qq{ak}^d}+{ad\qq{p} q^{\bb k}\over \qq{ak}^{d+1}}
\pmod{\qq{p}^2}.
\end{align}
Moreover
$$\sum_{k=1}^{p-1}{q^{b_1 j}\over \qq{j}^{d_1}}\cdot \sum_{k=1}^{p-1}{q^{b_2 k}\over \qq{k}^{d_2}}-\sum_{k=1}^{p-1}{q^{(b_1+b_2) k}\over \qq{k}^{d_1+d_2}}=
\sum_{1\leq j<k\leq p-1}{q^{b_1 j+b_2 k}\over \qq{j}^{d_1}\qq{k}^{d_2}}+\sum_{1\leq k<j\leq p-1}{q^{b_1 j+b_2 k}\over \qq{j}^{d_1}\qq{k}^{d_2}}$$
and by \eqref{qcp} we get
\begin{align*}
\sum_{1\leq k<j\leq p-1}{q^{b_1 j+b_2 k}\over \qq{j}^{d_1}\qq{k}^{d_2}}&=\sum_{1\leq j<k\leq p-1}{q^{b_1(p-j)+b_2(p- k)}\over \qq{p-j}^{d_1}\qq{p-k}^{d_2}}
\equiv \sum_{1\leq j<k\leq p-1}{(-1)^{d_1+d_2}q^{\bb_1 j+\bb_2 k}\over \qq{j}^{d_1}\qq{k}^{d_2}}  \pmod{\qq{p}}.
\end{align*}
Hence the proof of \eqref{qHH} is complete.
\end{proof} 

By letting $a=1$, $d=1$ and $b=0$ in \eqref{qHpos}, and by using
\eqref{qH}, we easily find \cite[Theorem 1]{SP}:
for any prime $p>3$:
$$\sum_{k=1}^{p-1}{1\over \qq{k}}=
{1\over 2}\sum_{k=1}^{p-1}{1-q^k\over \qq{k}}
{1\over 2}\sum_{k=1}^{p-1}{1+q^k\over \qq{k}}
\equiv {(1-q)(p-1)\over 2}+{(p^2-1)(1-q)^2\qq{p}\over 24}\pmod{\qq{p}^2}.$$
In a similar way, for $a=1$, $d=3$ and $b=1$, \eqref{qHpos} yields
\begin{align}\label{qHpos3}
\sum_{k=1}^{p-1} {q^k+q^{2k}\over \qq{k}^3}
&\equiv  (1-q)\qq{p}\sum_{k=1}^{p-1} {q^{2k}\over \qq{k}^3}-3\qq{p}\sum_{k=1}^{p-1} {q^{2k}\over \qq{k}^4}
\equiv -{\qq{p}(1-q)^4(p^4-1)\over 240}\pmod{\qq{p}^2}.
\end{align}

In order to show the $q$-congruences stated in the introduction, we need {\sl suitable} $q$-identities.
Such identities are not easy to find, but ones they are guessed correctly hopefully they can be proved via the $q$-WZ method
(see for example \cite{Mo,Ze}).

For $n\geq k\geq 0$, a pair $(F(n,k),G(n,k))$ is called {\sl $q$-WZ pair} if 
$$F(n+1,k)/F(n,k), \quad F(n,k+1)/F(n,k),\quad  G(n+1,k)/G(n,k),\quad  G(n,k+1)/G(n,k)$$
are all rational functions of $q^n$ and $q^k$, and
$$F(n+1,k)-F(n,k)=G(n,k+1)-G(n,k).$$
Let 
$$S(n)=\sum_{k=0}^{N-1} F(n,k)$$
then 
\begin{align*}
S(n+1)-S(n)&=F(n+1,n)+\sum_{k=0}^{n-1} (F(n+1,k)-F(n,k))\\
&=F(n+1,n)+\sum_{k=0}^{n-1} (G(n,k+1)-G(n,k))\\
&=F(n+1,n)+G(n,n)-G(n,0)
\end{align*}
and, by summing over $n$ from $0$ to $N-1$, we get the identity  
\begin{equation}\label{qMWZ}
\sum_{k=0}^{N-1} F(N,k)=\sum_{n=0}^{N-1}\left(F(n+1,n)+G(n,n)\right)-\sum_{n=0}^{N-1} G(n,0)
\end{equation}
which can be considered as the finite form of \cite[Theorem 7.3]{Mo}.

The $q$-identities we are interested in involves  {\sl Gaussian $q$-binomial coefficients}
$$\qbin{n}{k}=\left\{
\begin{array}{ll}
(q;q)_n(q;q)_k^{-1}(q;q)_{n-k}^{-1} &\mbox{if $0\leq k\leq n$},\\[3pt]
0 &\mbox{otherwise},
\end{array}\right.$$
where $(a;q)_n=\prod_{j=0}^{n-1}(1-aq^j)$ (note that $\qbin{n}{k}$ is a polynomial in $q$).
The next lemma allows us to reduce a special class of $q$-binomial coefficients 
modulo a power of $\qq{p}$.

\begin{lem} Let $p$ be a prime and let $a$ be a positive integer. For $k=1,\dots,p-1$ we have
\begin{align}
\label{bc1}&\qbin{ap-1}{k}\equiv(-1)^{k}q^{-{k+1\choose 2}}\left(1-a\qq{p}\sum_{j=1}^{k}{1\over \qq{j}}\right)\pmod{\qq{p}^2},\\
\label{bc2}&\qbin{p-1+k}{k}\equiv{\qq{p}\over \qq{k}}\left(1+\qq{p}\sum_{j=1}^{k-1}{q^j\over \qq{j}}\right)\pmod{\qq{p}^3},\\
\label{bc3}&
\qbin{p-1+k}{k}\qbin{p-1}{k}^{-1}\equiv {(-1)^{k}q^{{k+1\choose 2}}\qq{p}\over\qq{k}}\left(1+{\qq{p}\over \qq{k}}+\qq{p}\sum_{j=1}^{k-1}{1+q^j\over \qq{j}}\right)\pmod{\qq{p}^3}.
\end{align}
\end{lem}
\begin{proof} Since $\qq{ap}\equiv a\qq{p} \pmod{\qq{p}^2}$, it follows that
\begin{align*}
\qbin{ap-1}{k}
&=(-1)^{k}q^{-{k+1\choose 2}}\prod_{j=1}^{k}\left(1-{\qq{ap}\over \qq{j}}\right)
\equiv(-1)^{k}q^{-{k+1\choose 2}}\left(1-a\qq{p}\sum_{j=1}^{k}{1\over \qq{j}}\right)\pmod{\qq{p}^2}.
\end{align*}
Moreover
\begin{align*}
\qbin{p-1+k}{k}&={\qq{p}\over {\qq{k}}}\prod_{j=1}^{k-1}\left(1+{q^j\qq{p}\over \qq{j}}\right)
\equiv{\qq{p}\over \qq{k}}\left(1+\qq{p}\sum_{j=1}^{k-1}{q^j\over \qq{j}}\right)\pmod{\qq{p}^3}.
\end{align*}
Conguences \eqref{bc1} and \eqref{bc2} easily yield \eqref{bc3}.
\end{proof}
It should be noted that when $p$ is an odd prime, by using \eqref{bc2} for $k=p-1$, we recover the $q$-congruence \cite[(3.2)]{An2}:
\begin{equation}\label{qcan}
\qbin{ap-1}{p-1}\equiv q^{-{p\choose 2}}\left(1-a\qq{p}\sum_{j=1}^{p-1}{1\over \qq{j}}\right)
\equiv q^{-{p\choose 2}}\left(1-{a\qq{p}(p-1)(1-q)\over 2}\right)\equiv q^{(a-1){p\choose 2}}
\pmod{\qq{p}^2}.
\end{equation}

\section{Proof of \eqref{qc1}}
By \cite[(5.17)]{An1} (see \cite[(4.1)]{GZ} for a generalization), 
if $n$ is odd then by 
$$\sum_{k=0}^n {(-1)^{n-k}q^{{n-k\choose 2}}\over (-q;q)_k}\qbin{n}{k}\qbin{2k}{k}=0.$$
Hence for $n=p$ we have that
$$\sum_{k=1}^{p-1} {(-1)^{k-1}q^{{p-k\choose 2}-{p\choose 2}}\over (-q;q)_k\qq{k}}\qbin{p-1}{k-1}\qbin{2k}{k}={1\over (-q;q)_{p-1}\qq{p}}\left((-q;q)_{p-1}-q^{-{p\choose 2}}\qbin{2p-1}{p-1}\right).$$
By \eqref{bc1} and \eqref{qcan} we get
$$\sum_{k=1}^{p-1} {q^{k}\over (-q;q)_k\qq{k}}\qbin{2k}{k}\equiv\sum_{k=1}^{p-1} {q^{-pk+k}\over (-q;q)_k\qq{k}}\qbin{2k}{k}
\equiv{(-q;q)_{p-1}-1\over (-q;q)_{p-1}\qq{p}}\equiv Q_p(2;q)\pmod{\qq{p}},$$
and the first congruence is proved.

As regards the second one, we take
$$F(n,k)={(-1)^k\over (-q;q)_n\qq{k+1}}\qbin{n+k+1}{k+1}^{-1}\eee G(n,k)={q^{n+1}F(n,k)\over 1+q^{n+1}}.$$
This pair can be found in \cite{AZ}[Subsection 2.1] in connection with the 
irrationality proof of the $q$-series
$$\mbox{Ln}_q(2):=\sum_{k=1}^{\infty} {(-1)^{k}\over q^n-1}$$
for $|q|\not\in\{0,1\}$.
Hence by \eqref{qMWZ} we obtain the identity
\begin{equation}\label{id2}
\sum_{k=1}^n {(-1)^k(1+q^k+q^{2k})\over  (-q;q)_{k}\qq{k}}\qbin{2k}{k}^{-1}
={1\over (-q;q)_{n}}\sum_{k=1}^n {(-1)^k\over  \qq{k}}\qbin{n+k}{k}^{-1}
-\sum_{k=1}^n {q^k\over (-q;q)_k\qq{k}}
\end{equation}
Let $n=p-1$. Now by \eqref{bc1} and \cite[(1.5)]{Pa}
\begin{align*}
{1\over (-q;q)_{p-1}}\sum_{k=1}^{p-1} {(-1)^k\over  \qq{k}}\qbin{n+k}{k}^{-1}&\equiv \sum_{k=1}^{p-1} {(-1)^k\over  \qq{p}}\left(1- \qq{p}\sum_{j=1}^{k-1} {q^j\over \qq{j}}\right)\\
&\equiv -\sum_{k=1}^{p-1} (-1)^k\sum_{j=1}^{k-1} {q^j\over \qq{j}}\equiv -\sum_{j=1}^{p-1}{q^{2j-1}\over \qq{2j-1}} \\
&\equiv {(p-1)(1-q)\over 2}-\sum_{k=1}^{p-1}{1\over \qq{j}}+\sum_{k=1}^{(p-1)/2}{1\over \qq{2j}}\equiv -Q_p(2;q)\pmod{\qq{p}}.
\end{align*}
By the $q$-binomial theorem,
$$\sum_{k=0}^{n} \qbin{n}{k}\prod_{j=0}^{k-1}(x-q^j)=x^n$$
and for $n=p$, $x=-1$ together with \eqref{bc1}, we get
$$\sum_{k=1}^{p-1} {q^{-{k\choose 2}}(-q;q)_{k-1}\over \qq{k}}
\equiv\sum_{k=1}^{p-1} {(-1)^{k-1}(-q;q)_{k-1}\over \qq{k}}\qbin{p-1}{k-1}=-Q_p(2;q).$$
Hence (see the {\sl dual} congruence \cite[(5.4)]{Pa})
$$\sum_{k=1}^{p-1} {q^k\over (-q;q)_k\qq{k}}\equiv \sum_{k=1}^{p-1} {q^{p-k}\over (-q;q)_{p-k}\qq{p-k}}
\equiv \sum_{k=1}^{p-1} {q^{-{k\choose 2}}(-q;q)_{k-1}\over \qq{k}}
\equiv -Q_p(2;q)\pmod{\qq{p}}$$
where we used
$$\qq{p-k}=-q^{-k}\qq{k} \eee (-q;q)_{p-k}^{-1}\equiv q^{-{k\choose 2}}(-q;q)_{k-1}\pmod{\qq{p}}.$$
Therefore, by identity \eqref{id2},
$$\sum_{k=1}^{p-1} {(-1)^k(1+q^k+q^{2k})\over  (-q;q)_{k}\qq{k}}\qbin{2k}{k}^{-1}\equiv -2Q_p(2;q) \pmod{\qq{p}}$$
and we are done.
\section{Proof of \eqref{qc2}}
Let
$$F(n,k)={q^{-{k+1\choose 2}}\over \qq{k+1}}\qbin{n+k+2}{n+1}\qbin{n+1}{k+1}^{-1}$$
and
$$G(n,k)=-{q^{n+2}(1-q^{k+1})(1-q^{n+1-k})F(n,k)\over (1+q^{n+2})(1-q^{n+2})^2}$$
then \eqref{qMWZ} gives the identity
\begin{equation}\label{id3}
\sum_{k=1}^n {(1+q^k+q^{2k})\,q^{-{k\choose 2}}\over (1+q^k)^2\qq{k}}\qbin{2k}{k}
=\sum_{k=1}^n {q^{-{k\choose 2}}\over \qq{k}}\qbin{n+k}{k}\qbin{n}{k}^{-1}
-\sum_{k=1}^n {q^k\over \qq{2k}}.
\end{equation}
Let $n=p-1$ then by \eqref{bc3} and \eqref{qHneg}, we obtain
$$\sum_{k=1}^{p-1} {q^{-{k\choose 2}}\over \qq{k}}\qbin{p-1+k}{k}\qbin{p-1}{k}^{-1}
\equiv  \qq{p}\sum_{k=1}^{p-1} {(-1)^{k}q^{k}\over \qq{k}^2}\equiv 0\pmod{\qq{p}^2}.$$
Moreover, \eqref{qHpos} and \eqref{qH2} implies that
$$\sum_{k=1}^{p-1} {q^k\over \qq{2k}}\equiv -\qq{p}\sum_{k=1}^{p-1} {q^k\over \qq{2k}^2}\equiv - {\qq{p}(p^2-1)(1-q)^2\over 24}\pmod{\qq{p}^2},$$
and \eqref{qc2} follows easily from \eqref{id3}.

\noindent Note that by letting $q\to 1$ in \eqref{id3} we obtain the identity
$${3\over 4}\sum_{k=1}^n {1\over k}{2k \choose k}
=\sum_{k=1}^n {1\over k}\binom{n+k}{k}\binom{n}{k}^{-1}-{1\over 2}\sum_{k=1}^n {1\over k}.$$
which can be exploited to prove an improvement of \cite[Theorem 4.2]{Ta2}:
$$\sum_{k=1}^{p-1} {1\over k}{2k \choose k}\equiv -{8\over 3}H_{p-1}(1)+2p^4 B_{p-5}\pmod{p^5}$$
for any prime $p>3$.
\section{Proof of \eqref{qc3}}
By  taking 
$$F(n,k)={(-1)^k q^{-{k+1\choose 2}}(1+q^{k+1})\over \qq{k+1}^2}\qbin{n+k+2}{n+1}\qbin{n+1}{k+1}^{-1}$$
and
$$G(n,k)={q^{n+2}(1-q^{k+1})^2(1-q^{n+1-k})F(n,k)\over (1+q^{k+1})(1-q^{n+2})^3},$$
\eqref{qMWZ} yields the identity
\begin{equation}\label{id4}
\sum_{k=1}^n {(-1)^k(1+3q^k+q^{2k})q^{-{k\choose 2}}\over (1+q^k)\qq{k}^2}\qbin{2k}{k}
=\sum_{k=1}^n {(-1)^k(1+q^k)q^{-{k\choose 2}}\over \qq{k}^2}\qbin{n+k}{k}\qbin{n}{k}^{-1}
-\sum_{k=1}^n {q^k\over \qq{k}^2}.
\end{equation}
Let $n=p-1$. Now by \eqref{bc3}
\begin{align*}
\sum_{k=1}^{p-1} {(-1)^k(1+q^k)q^{-{k\choose 2}}\over \qq{k}^2}\qbin{p-1+k}{k}\qbin{p-1}{k}^{-1}
&\equiv  \qq{p}\sum_{k=1}^{p-1} {q^k+q^{2k}\over \qq{k}^3}+\qq{p}^2\sum_{k=1}^{p-1} {q^k+q^{2k}\over \qq{k}^4}\\
&\qquad +\qq{p}^2\sum_{1\leq j<k\leq p-1} {(1+q^j)(q^k+q^{2k})\over \qq{j}\qq{k}^3}
\pmod{\qq{p}^3}.
\end{align*}
Then the $q$-congruence \eqref{qc3} follows from \eqref{id4} by using \eqref{qHpos3} and
$$\sum_{k=1}^{p-1} {q^k+q^{2k}\over \qq{k}^4}\equiv -\sum_{1\leq j<k\leq p-1} {(1+q^j)(q^k+q^{2k})\over \qq{j}\qq{k}^3}\equiv
{(1-q)^4(p^2-1)(p^2-4)\over 360}\pmod{\qq{p}},$$
which is a straightforward application of \eqref{qH} and \eqref{qHH}.


\vspace*{0.2cm}


\begin{thebibliography}{99}
\bibitem{AZ} T. Amdeberhan, D. Zeilberger,
{\it $q$-Ap\'ery irrationality proofs by $q$-WZ pairs},
Adv. in Appl. Math. {\bf 20} (1998), 275--283.

\bibitem{An1} G. E. Andrews,
{\it Applications of basic hypergeometric functions},
SIAM Rev. {\bf 16} (1974), 441--484.

\bibitem{An2} G. E. Andrews,
{\it $q$-Analogs of the binomial coefficient congruences of Babbage, Wolstenholme and Glaisher},
Discrete Math. {\bf 204} (1999), 15--25.


\bibitem{Di} K. Dilcher,
{\it Determinant expressions for $q$-harmonic congruences and degenerate Bernoulli numbers},
Electron. J. Combin. {\bf 15} (2008), 1--18.

\bibitem{GZ} V. J. W. Guo, J. Zeng,
{\it Some congruences involving central $q$-binomial coefficients},
Adv. in Appl. Math. {\bf 45} (2010), 303--316.

\bibitem{Mo} M. Mohammed,
{\it The $q$-Markov-WZ method},
Ann. Comb. {\bf 9} (2005), 205-221.

\bibitem{Pa} H. Pan,
{\it A $q$-analogue of Lehmer's congruence},
Acta Arith. {\bf 128} (2007), 303--318.

\bibitem{PP} Kh. Hessami Pilehrood, T. Hessami Pilehrood,
{\it A $q$-analogue of the Bailey-Borwein-Bradley identity},
J. Symbolic Comput. {\bf 46} (2011), 699--711.

\bibitem{SP} L. L. Shi, H. Pan,
{\it A $q$-analogue of Wolstenholme's harmonic series congruenc},
Amer. Math. Monthly {\bf 114} (2007), 529--531.



\bibitem{ST2} Z. W. Sun, R. Tauraso, 
{\it New congruences for central binomial coefficients}, 
Adv. in Appl. Math. {\bf 45} (2010), 125--148.

\bibitem{Ta1} R. Tauraso, 
{\it $q$-Analogs of some congruences involving Catalan numbers},
to appear in Adv. in Appl. Math.

\bibitem{Ta2} R. Tauraso, 
{\it More congruences for central binomial coefficients},
J. Number Theory {\bf 130} (2010), 2639--2649.

\bibitem{Ze} D. Zeilberger, 
{\it Closed form (pun intended!)},
Contemp. Math. {\bf 143} (1993), 579--607.
\end{thebibliography}
\end{document}